\documentclass[10pt]{amsart}
\usepackage{bm}

\usepackage{amsmath}
\usepackage{amsthm}
\usepackage{amssymb}

\usepackage{hyperref}

\newcommand{\R}{\mathbb{R}}

\newcommand{\N}{\mathbb{N}}
\newcommand{\Z}{\mathbb{Z}}
\newcommand{\lr}[1]{\langle #1 \rangle}
\newcommand{\eps}{\varepsilon}

\newcommand{\Lc}{\mathcal{L}}

\newcommand{\lp}{\mathcal{U}}

\newcommand{\uf}{\mathfrak{u}}
\newcommand{\vf}{\mathfrak{v}}
\newcommand{\wf}{\mathfrak{w}}

\newcommand{\J}{\mathfrak{J}}
\newcommand{\K}{\mathfrak{K}}

\newcommand{\dx}{\partial_x}
\newcommand{\dy}{\partial_y}
\newcommand{\dt}{\partial_t}

\DeclareMathOperator*{\supp}{supp}

\renewcommand{\H}{\mathcal{H}}

\newcommand{\wt}[1]{\widetilde{#1}}

\newcommand{\wh}[1]{\widehat{#1}}

\newtheorem{thm}{Theorem}[section]
\newtheorem{prop}[thm]{Proposition}

\newtheorem{lem}[thm]{Lemma}

\theoremstyle{remark}
\newtheorem{rmk}[thm]{Remark}

\numberwithin{equation}{section}

\title[WP for KdV-type eqs]{Well-posedness for KdV-type equations \\ with quadratic nonlinearity}

\author[H. Hirayama]{Hiroyuki Hirayama}
\address[H. Hirayama]{Organization for Promotion of Tenure Track, University of Miyazaki, 1-1, Gakuenkibanadai-nishi, Miyazaki, 889-2192 Japan}
\email{h.hirayama@cc.miyazaki-u.ac.jp}

\author[S. Kinoshita]{Shinya Kinoshita}
\address[S. Kinoshita]{Universit\"at Bielefeld, 
Fakult\"at f\"ur Mathematik, Postfach 10 01 31 33501, Bielefeld, Germany}
\email{kinoshita@math.uni-bielefeld.de}

\author[M. Okamoto]{Mamoru Okamoto}
\address[M. Okamoto]{Division of Mathematics and Physics, Faculty of Engineering, Shinshu University, 4-17-1 Wakasato, Nagano City 380-8553, Japan}
\email{m\_okamoto@shinshu-u.ac.jp}



\date{\today}

\allowdisplaybreaks[1]
\begin{document}

\begin{abstract}
We consider the Cauchy problem of the KdV-type equation
\[
\dt u + \frac{1}{3} \dx^3 u = c_1 u \dx^2u + c_2 (\dx u)^2, \quad u(0)=u_0.
\]
Pilod (2008) showed that the flow map of this Cauchy problem fails to be twice differentiable in the Sobolev space $H^s(\R)$ for any $s \in \R$ if $c_1 \neq 0$.
By using a gauge transformation, we point out that the contraction mapping theorem is applicable to the Cauchy problem if the initial data are in $H^2(\R)$ with bounded primitives.
Moreover, we prove that the Cauchy problem is locally well-posed in $H^1(\R)$ with bounded primitives.
\end{abstract}

\maketitle

\section{Introduction}

We consider the Cauchy problem for the Korteweg-de Vries (KdV) type equation
\begin{equation} \label{KdV}
\dt u + \frac{1}{3} \dx^3 u = c_1 u \dx^2u + c_2 (\dx u)^2,
\end{equation}
where $u$ is a real valued function and $c_1$ and $c_2$ are real constants.

If $c_1=0$, because $\dx u$ satisfies the KdV equation,
the results by Kenig et al. \cite{KPV96} and Kishimoto \cite{Kis09} imply that \eqref{KdV} is well-posed in the Sobolev space $H^s (\R)$ for $s \ge \frac{1}{4}$.
On the other hand, Tarama \cite{Tar95} proved that even a linear equation requires a Mizohata-type condition for the well-posedness in $L^2(\R)$ (see also \cite{Miz85}).
Indeed, the linear equation
\[
(\dx + \dx^3 + a(x) \dx^2) u =0
\]
where $a$ is smooth with bounded derivatives is well-posed in $L^2(\R)$ if and only if
\[
\sup_{x_1 \le x_2} \int_{x_1}^{x_2} a(x) dx <\infty
\]
holds.
Hence, at least, well-posedness in $H^s(\R)$ for \eqref{KdV} requires some additional conditions.
In fact, Pilod \cite{Pil08} showed that the flow map of this Cauchy problem fails to be twice differentiable in $H^s(\R)$ for any $s \in \R$ if $c_1 \neq 0$.

Local well-posedness was established using the weighted Sobolev spaces $H^s(\R) \cap L^2(x^{2k}dx)$ for sufficiently large $s$ and $k$ by Kenig et al. \cite{KPV94} and Kenig and Staffilani \cite{KenSta97}.
For the proof, they used a change of dependent variables as in \cite{HayOza92, HayOza94}.
In these works, the change of dependent variable was called a gauge transformation.
By replacing weighted spaces with a spatial summability condition, Harrop-Griffiths \cite{Har-G15I} proved local well-posedness for \eqref{KdV} in a translation invariant space $l^1 H^s (\R)$ for $s>\frac{5}{2}$.
We note that he also treated more general semi-linear nonlinearity (see also \cite{Har-G15II}).

We mention the well-posedness results for the third-order Benjamin-Ono equation
\begin{equation} \label{3BO}
\dt u - b \H \dx^2 u + a \dx^3u = c u^2 \dx u - \dx ( u \H \dx u + \H ( u \dx u)),
\end{equation}
where $\H$ is the Hilbert transform, $a,b,c$ are constants with $a \neq 0$ and $b,c \ge 0$.
As in \eqref{KdV}, local well-posedness in $H^s(\R)$ for \eqref{3BO} cannot be established by an iteration argument.
When $b=0$, Feng and Han \cite{FenHan96} performed the energy estimate and proved the existence of a unique global solution in $H^s(\R)$ for $s \ge 4$ (see also \cite{Fen97}).
By using a gauge transformation as in \cite{HayOza92, HayOza94}, Linares et al. \cite{LPP11} proved that the Cauchy problem for \eqref{3BO} with $c=0$ is locally well-posed in $H^s(\R)$ with $s \ge 2$ or $H^k (\R) \cap L^2(x^2 dx)$ with $k \in \Z_{\ge 2}$.
Molinet and Pilod \cite{MolPil12} showed that the global-in-time well-posedness in $H^1(\R)$.
In addition to the gauge transformation, they used the Fourier restriction norm to show an a priori estimate in $H^1(\R)$.

In this paper, by using a gauge transformation as in \cite{Oza98}, we show the well-posedness for \eqref{KdV} in the Sobolev spaces with bounded primitives.
We define the function space
\[
\mathcal{X}^s := \left\{ f \in H^s (\R) \colon \sup_{x \in \R} \left| \int_{-\infty}^x f(y) dy \right| <\infty \right\}
\]
for $s \in \R$.
This space is a Banach space equipped with the norm
\[
\| f \|_{\mathcal{X}^s} := \| f \|_{H^s} + \left\| \int_{-\infty}^x f(y) dy \right\|_{L_x^{\infty}}
\]
for $s > \frac{1}{2}$ (see Proposition 1 in \cite{Oza98}).
The following is our main result.

\begin{thm} \label{WP}
The Cauchy problem for \eqref{KdV} with $u(0) =u_0$ is local-in-time well-posed in $\mathcal{X}^s$ for $s \ge 1$.
Moreover, the flow map is (locally) Lipschitz continuous.
In addition, the existence time depends only on $\| u_0 \|_{\mathcal{X}^1}$.
\end{thm}

\begin{rmk}
We note that $l^1 H^s(\R)$ is embedded in $H^s(\R)$ and that $s>1$ yields $l^1 H^s(\R) \hookrightarrow L^1(\R)$.
Hence, our functions space $\mathcal{X}^s$ is bigger than $l^1 H^s(\R)$, indeed
\[
\sup_{x \in \R} \left| \int_{-\infty}^x f(y) dy \right|
\le \| f \|_{L^1}
\lesssim \| f \|_{l^1 H^s}
\]
holds provided that $s>1$.
Moreover, the function $f(x) = \frac{\sin x}{x}$ is an example that $f \in \mathcal{X}^s$ for any $s \in \R$, but $f \notin l^1H^s(\R)$.
In the quadratic setting, our result is an improvement of that in \cite{Har-G15I} from the view point both of the integrability and the regularity.
\end{rmk}

For the proof, we use a gauge transformation as in \cite{Oza98}, which makes \eqref{KdV} a coupled system of KdV-type equations (see \eqref{eq:u} and \eqref{eq:vf} below).
Roughly speaking, the gauge transformation for \eqref{KdV} and \eqref{3BO} is defined as
\[
u \mapsto e^{\int_{-\infty}^x u(t,y)dy} u, \quad
u \mapsto e^{i \int_{-\infty}^x u(t,y)dy} u,
\]
respectively.
Thanks to the presence of $i$, the $L^2$-norm is invariant under the gauge transformation for \eqref{3BO}.
On the other hand, the $L^2$-boundedness of the gauge transformation for \eqref{KdV} requires that the primitives of $u$ are bounded.

Here, we give an outline of the proof of Theorem \ref{WP}.
Our proof depends on the gauge transformation but not on the energy estimate and the Fourier restriction norm.
To calculate the nonlinear terms, we use the Strichartz estimate, the local smoothing estimate, and the maximal function estimate.

We apply the gauge transformation to rewrite \eqref{KdV} to a coupled system of KdV-type equations as mentioned above.
First, by using the contraction mapping theorem, we show that the system is well-posed in $\mathcal{X}^1 \times H^1$ in \S \ref{S:WPX0}, which yields that \eqref{KdV} is well-posed in $\mathcal{X}^2$.
Second, we prove the a priori estimate \eqref{apriori000} in \S \ref{S:WPX1}, which says that the existence time depends only on $\| u_0 \|_{\mathcal{X}^1}$ as long as $u$ is a solution to \eqref{KdV}.
Therefore, Theorem \ref{WP} with $s=1$ follows from an approximation argument and the fact that the solution to \eqref{KdV} exists at least in $\mathcal{X}^2$.
Because the well-posedness in $\mathcal{X}^2$ is required only in this approximation argument, we may use the result in \cite{Har-G15I} instead of the well-posedness in $\mathcal{X}^2$.
However, for a self-contained proof of Theorem \ref{WP}, we employ the well-posedness in $\mathcal{X}^2$.
Third, by applying the fractional Leibniz rule as in \cite{KPV93}, we show the well-posedness in $\mathcal{X}^s$ for $s \ge 1$ and the persistence property in \S \ref{S:WPX12}.

We observe that $\| u \|_{L_x^2 L_T^{\infty}}$ is bounded by the norms of $u$ and the gauge transformed $u$ (Lemma \ref{lem:unbound}).
Because the quadratic term with derivative in \eqref{eq:vf} vanishes when $c_2=0$, the a priori bound \eqref{apriori000} (see also \eqref{apriori00}) follows from these facts and a similar argument as in \S \ref{S:WPX0}.
For $c_2 \neq 0$, by using a gauge transformation, we rewrite \eqref{KdV} to an equation which contains no terms of the form $(\dx u)^2$.
Namely, we apply the gauge transformation twice to obtain Theorem \ref{WP} in general.
This is the reason why we can avoid using the Fourier restriction norm.

Our argument can estimate the difference of two solutions to \eqref{KdV}, and hence the flow map is (locally) Lipschitz continuous.
On the other hand, the flow map is not smooth for low-regularity data even with bounded primitives.

\begin{prop}
If $s<1$, then the flow map of \eqref{KdV} fails to be twice differentiable in $\mathcal{X}^s$.
\end{prop}

We also consider a semi-linear KdV-type equation with quadratic nonlinearity
\begin{equation} \label{KdV2}
\dt u + \frac{1}{3} \dx^3 u = c_1 u \dx^2u + c_2 (\dx u)^2 + c_3 \dx u \dx^2 u + c_4 (\dx^2 u)^2.
\end{equation}
Because $\dx^2u$ ($\dx u$ if $c_4=0$) satisfies an equation like as \eqref{KdV}, the same argument as in the proof of Theorem \ref{WP} yields the following:

\begin{thm} \label{WP2}
The Cauchy problem for \eqref{KdV2} with $u(0) =u_0$ is local-in-time well-posed in $\mathcal{X}^3$.
Moreover, we can replace $\mathcal{X}^3$ by $\mathcal{X}^2$ if $c_4=0$.
In addition, the persistence of regularity holds.
\end{thm}

\begin{rmk}
We can remove the boundedness of primitives if $c_1=0$.
More precisely, the Cauchy problem for \eqref{KdV2} is well-posed in $H^2(\R)$ and $H^3(\R)$ provided that $c_1=c_4=0$ and $c_1=0$, respectively.
\end{rmk}

\subsection{Notation}

We denote the set of nonnegative integers by $\N_0 := \N \cup \{ 0 \}$.
Let $P_N$ denote the (inhomogeneous) Littlewood-Paley decomposition:
\[
u = \sum_{N \in 2^{\N_0}} P_N u.
\]
Let $1 \le p,q \le \infty$ and $T>0$.
Define
\begin{align*}
& \| f \|_{L_x^p L_T^q} := \left( \int_{-\infty}^{\infty} \left( \int_{-T}^T |f(t,x)|^q dt \right)^{\frac{p}{q}} dx \right)^{\frac{1}{p}},
\\
& \| f \|_{L_T^q L_x^p} := \left( \int_{-T}^T \left( \int_{-\infty}^{\infty} |f(t,x)|^p dx \right)^{\frac{q}{p}} dt \right)^{\frac{1}{q}},
\end{align*}
with $T=t$ to indicate the case when $T=\infty$.

We set $\Lc := \dt + \frac{1}{3} \dx^3$.
Let $\lp (t)$ be the linear propagator of \eqref{KdV}, that is $\lp (t) := e^{-\frac{t}{3} \dx^3}$.

In estimates, we use $C$ to denote a positive constant that can change from line to line.
We write $A\lesssim B$ to mean $A \le CB$ if $C$ is absolute or depends only on parameters that are considered fixed.
We define $A \ll B$ to mean $A \le C^{-1} B$.

\section{Lemmas}

In this section, we collect some lemmas which are used in the proof.

The first lemma is the Strichartz estimate for the Airy equation.

\begin{lem}[Lemma 2.4 in \cite{KPV89}] \label{lem:Str}
Let $2 \le q,r \le \infty$ and $0 \le s \le \frac{1}{q}$ satisfy $-s+\frac{3}{q}+\frac{1}{r}=\frac{1}{2}$.
Then,
\[
\left\| |\dx|^s \lp (t) u_0 \right\|_{L_t^q L_x^r} \lesssim \| u_0 \|_{L^2}.
\]
\end{lem}

The second lemma is the local smoothing effect of Kato-type (see, for example, Theorem 3.5 in \cite{KPV93}).

\begin{lem} \label{lem:Kato}
For any $u_0 \in L^2(\R)$, we have
\[
\| \dx \lp (t) u_0 \|_{L_x^{\infty} L_t^2} \lesssim \| u_0 \|_{L^2}.
\]
\end{lem}

The third lemma is the maximal function estimates.

\begin{lem}[Corollary 2.9 in \cite{KPV91}] \label{lem:KL} 
Let $s>\frac{3}{4}$.
Then for any $u_0 \in H^s(\R)$ and any $\rho>\frac{3}{4}$,
\[
\| \lp (t) u_0 \|_{L_x^2 L_T^{\infty}} \lesssim \lr{T}^{\rho} \| u_0 \|_{H^s}.
\]
\end{lem}

\section{Well-posedness via the contraction mapping theorem} \label{S:WPX0}

In this section, by using the iteration argument, we show that \eqref{KdV} is locally well-posed in $\mathcal{X}^2$.

First, we observe some formal calculations.
Let $\Lambda$ and $v$ be real valued functions.
A direct calculation shows
\begin{equation} \label{gauge}
e^{\Lambda} \Lc \left( e^{-\Lambda} v\right)
= \Lc v + \left( \dx \Lambda \dx^2 \Lambda -\frac{1}{3} (\dx \Lambda)^3 - \Lc \Lambda \right) v + \left( -\dx^2 \Lambda + (\dx \Lambda)^2 \right) \dx v - \dx \Lambda \dx^2 v.
\end{equation}
Let $u$ be a solution to \eqref{KdV} and set $v= \dx u$.
Then, \eqref{KdV} yields
\[
\Lc v = \dx \Lc u = (c_1+2c_2) \dx u \dx^2 u + c_1 u \dx^3 u.
\]
To cancel out the worst part, we set $\Lambda (t,x) =c_1 \int_{-\infty}^x u(t,y)dy$.
Since
\begin{equation} \label{LLambda}
\begin{aligned}
\Lc \Lambda
&= c_1 \int_{-\infty}^x (\Lc u)(t,y) dy
= c_1^2 \int_{-\infty}^x (u \dy^2 u) dy + c_1c_2 \int_{-\infty}^x (\dy u)^2 dy \\
&= c_1^2 u \dx u + c_1(-c_1+c_2) \int_{-\infty}^x (\dy u)^2 dy,
\end{aligned}
\end{equation}
\eqref{gauge} with $v=\dx u$ leads to the following:
\[
e^{\Lambda} \Lc \left( e^{-\Lambda} \dx u \right)
= 2c_2 \dx u \dx^2 u + c_1^2 u^2 \dx^2 u+ c_1(c_1-c_2) \dx u \int_{-\infty}^x (\dy u)^2 dy - \frac{c_1^3}{3} u^3 \dx u.
\]
Hence, by setting $\vf := e^{-\Lambda} \dx u$, we have
\begin{align}
&\Lc u = c_1 e^{\Lambda} u \left( \dx \vf + c_1 u \vf \right) + c_2 e^{2\Lambda} \vf^2, \label{eq:u} \\
\label{eq:vf}
&\Lc \vf
= 2 c_2 e^{\Lambda} \vf \left( \dx \vf + c_1 u \vf \right) + c_1^2 u^2 \dx \vf + c_1 (c_1-c_2) \vf \int_{-\infty}^x e^{2\Lambda} \vf^2 dy + \frac{2}{3} c_1^3 u^3 \vf.
\end{align}

\subsection{Proof of Theorem \ref{WP} with $s=2$}

Let $\eps>0$ be sufficiently small.
We define the function space $X_T$ for $T>0$ by
\begin{align*}
X_T &:= \{ f \in L^{\infty} ([-T,T];L^2 (\R)): \| f \|_{X_T} <\infty \}, \\
\| f \|_{X_T} &:= \| f \|_{L_T^{\infty} L_x^2} + \| f \|_{L_T^6 L_x^{\infty}} + \| \dx f \|_{L_x^{\infty} L_T^2} + \left\| \lr{\dx}^{-\frac{3}{4}-\eps} f \right\|_{L_x^2 L_T^{\infty}}.
\end{align*}
Lemmas \ref{lem:Str}--\ref{lem:KL} yield that
\begin{equation} \label{est:lin}
\| \lp (t) u_0 \|_{X_T} \le C_1 \| u_0 \|_{L^2}
\end{equation}
for $0<T<1$.
In addition, an interpolation shows that
\[
\| u \|_{L_T^q L_x^r} \lesssim \| u \|_{X_T}
\]
for any $2 \le q,r \le \infty$ with $\frac{3}{q}+\frac{1}{r}=\frac{1}{2}$.
In particular, $(q,r)= (12,4), (9,6), (8,8)$ are allowed.
Furthermore, for such $(q,r)$, $1\le q'<q$, and $0<T<1$, we have
\[
\|u\|_{L^{q'}_TL^r_x}\le T^{\frac{1}{q'}-\frac{1}{q}}\|u\|_{L^q_TL^r_x}
\lesssim \|u\|_{X_T}.
\]

We will apply the contraction mapping theorem in the space
\[
Y_T := \left\{ (u,\vf) \in X_T \times X_T \colon \lr{\dx}u \in X_T, \, \lr{\dx} \vf \in X_T, \, \sup_{\substack{|t|\le T \\ x \in \R}} \left| \int_{-\infty}^x u(t,y) dy \right| <\infty \right\}
\]
equipped with the norm
\[
\| (u, \vf) \|_{Y_T} := \| \lr{\dx} u \|_{X_T} + \| \lr{\dx} \vf \|_{X_T} + \sup_{|t| \le T, x \in \R} \left| \int_{-\infty}^x u(t,y) dy \right|.
\]

We define $\Psi_{u_0} (u,\vf) := \left( \Psi_{u_0}^{(1)} (u,\vf), \Psi_{u_0}^{(2)} (u,\vf) \right)$ by
\begin{align*}
\Psi_{u_0}^{(1)} (u,\vf) &:= \lp (t) u_0 + \int_0^t \lp (t-t') \left\{ c_1 e^{\Lambda} u \left( \dx \vf + c_1 u \vf \right) + c_2 e^{2\Lambda} \vf^2 \right\} (t',x) dt', \\
\Psi_{u_0}^{(2)} (u,\vf) &:= \lp (t) \vf_0 + \int_0^t \lp (t-t') \bigg\{ 2 c_2 e^{\Lambda} \vf \left( \dx \vf + c_1 u \vf \right)+ c_1^2 u^2 \dx \vf \\
&\qquad  + c_1 (c_1-c_2) \vf \int_{-\infty}^x e^{2\Lambda} \vf^2 dy + \frac{2}{3} c_1^3 u^3 \vf \bigg\} (t',x) dt',
\end{align*}
where $\Lambda (t,x) := c_1 \int_{-\infty}^x u (t,y)dy$ and $\vf_0 := e^{-c_1 \int_{-\infty}^x u_0(y) dy} \dx u_0$.

Let $0<T < 1$ be determined later.
Then, H\"older's inequality yields that
\begin{align*}
&
\begin{aligned}
\| u \dx \vf \|_{L_T^2 H_x^1}
&\lesssim \| u \dx^2 \vf \|_{L_{T,x}^2} + \| \dx u \dx \vf \|_{L_{T,x}^2} + \| u \dx \vf \|_{L_{T,x}^2}\\
&\lesssim \| u \|_{L_x^2 L_T^{\infty}} \| \dx^2 \vf \|_{L_x^{\infty} L_T^2} + \| \dx u \|_{L_{T,x}^4} \| \dx \vf \|_{L_{T,x}^4} + \| u \|_{L_x^2 L_T^{\infty}} \| \dx \vf \|_{L_x^{\infty} L_T^2} \\
&\lesssim \| \lr{\dx} u \|_{X_T} \| \lr{\dx} \vf \|_{X_T},
\end{aligned}
\\
&
\begin{aligned}
\| u^2 \vf \|_{L_T^2 H_x^1}
&\lesssim \| u^2 \dx \vf \|_{L_{T,x}^2} + \| u \dx u \vf \|_{L_{T,x}^2} + \| u^2 \vf \|_{L_{T,x}^2} \\
&\lesssim \| u \|_{L_{T,x}^{\infty}} \| u \|_{L_x^2 L_T^{\infty}} \| \dx \vf \|_{L_x^{\infty} L_T^2} + \| u \|_{L_{T,x}^{\infty}} \| \lr{\dx} u \|_{L_{T,x}^4} \| \vf \|_{L_{T,x}^4} \\
&\lesssim \| \lr{\dx} u \|_{X_T}^2 \| \lr{\dx} \vf \|_{X_T},
\end{aligned}
\\
&
\begin{aligned}
\| \vf^2 \|_{L_T^2 H_x^1}
\lesssim \| \vf \dx \vf \|_{L_{T,x}^2} + \| \vf^2 \|_{L_{T,x}^2}
\lesssim \| \vf \|_{L_{T,x}^{\infty}} \| \lr{\dx} \vf \|_{L_T^{\infty} L_x^2}
\lesssim \| \lr{\dx} \vf \|_{X_T}^2.
\end{aligned}
\end{align*}
Since
\begin{equation} \label{product}
\| e^{\Lambda} f \|_{H^1}
\lesssim \| e^{\Lambda} \|_{L_{T,x}^{\infty}} \left( \| u \|_{L_{T,x}^{\infty}}+1 \right) \| f \|_{H^1}
\lesssim \| e^{\Lambda} \|_{L_{T,x}^{\infty}} \left( \| \lr{\dx} u \|_{X_T}+1 \right)  \| f \|_{H^1},
\end{equation}
we use \eqref{est:lin} to obtain the following:
\begin{equation} \label{Phiest1}
\begin{aligned}
&\| \lr{\dx} \Phi_{u_0}^{(1)} (u,\vf) \|_{X_T} - C_1 \| u_0 \|_{H^1} \\
&\lesssim T^{\frac{1}{2}} \left( \| e^{\Lambda} u \dx \vf \|_{L_T^2 H_x^1} + \| e^{\Lambda} u^2 \vf \|_{L_T^2 H_x^1} + \| e^{2\Lambda} \vf^2 \|_{L_T^2 H_x^1} \right) \\
&\lesssim T^{\frac{1}{2}} \left( \| e^{\Lambda} \|_{L_{T,x}^{\infty}} + \| e^{2\Lambda} \|_{L_{T,x}^{\infty}} \right) \left( \| (u,\vf) \|_{Y_T}^2 + \| (u,\vf) \|_{Y_T}^4 \right).
\end{aligned}
\end{equation}
Moreover, we observe the following estimates:
\begin{align*}
&
\begin{aligned}
\| \vf \dx \vf \|_{L_T^2 H_x^1}
&\lesssim \| \vf \dx^2 \vf \|_{L_{T,x}^2} + \| (\dx \vf)^2 \|_{L_{T,x}^2} + \| \vf \dx \vf \|_{L_{T,x}^2} \\
&\lesssim \| \vf \|_{L_x^2 L_T^{\infty}} \| \dx^2 \vf \|_{L_x^{\infty} L_T^2} + \| \dx \vf \|_{L_{T,x}^4}^2 + \| \vf \|_{L_{T,x}^{\infty}} \| \dx \vf \|_{L_T^{\infty} L_x^2} \\
&\lesssim \| \lr{\dx} \vf \|_{X_T}^2,
\end{aligned}
\\
&
\begin{aligned}
&\| u \vf^2 \|_{L_T^2 H_x^1} \\
&\quad \lesssim \| u \vf \dx \vf \|_{L_{T,x}^2} + \| \dx u \vf^2 \|_{L_{T,x}^2} + \| u \vf^2 \|_{L_{T,x}^2} \\
&\quad \lesssim \| \vf \|_{L_{T,x}^{\infty}} \left( \| u \|_{L_x^2 L_T^{\infty}} \| \dx \vf \|_{L_x^{\infty} L_T^2} + \| \dx u \|_{L_x^{\infty} L_T^2} \| \vf \|_{L_x^2 L_T^{\infty}} + \| u \|_{L_{T,x}^{\infty}} \| \vf \|_{L_T^{\infty} L_x^2} \right) \\
&\quad \lesssim \| \lr{\dx} u \|_{X_T} \| \lr{\dx} \vf \|_{X_T}^2,
\end{aligned}
\\
&
\begin{aligned}
&\| u^2 \dx \vf \|_{L_T^2 H_x^1} \\
&\quad \lesssim \| u^2 \dx^2 \vf \|_{L_{T,x}^2} + \| u \dx u \dx \vf \|_{L_{T,x}^2} + \| u^2 \dx \vf \|_{L_{T,x}^2} \\
&\quad \lesssim \| u \|_{L_{T,x}^{\infty}} \left( \| u \|_{L_x^2 L_T^{\infty}} \| \dx^2 \vf \|_{L_x^{\infty} L_T^2} + \| \dx u \|_{L_{T,x}^4} \| \dx \vf \|_{L_{T,x}^4} + \| u \|_{L_{T,x}^{\infty}} \| \dx \vf \|_{L_T^{\infty} L_x^2} \right) \\
&\quad \lesssim \| \lr{\dx} u \|_{X_T}^2 \| \lr{\dx} \vf \|_{X_T},
\end{aligned}
\\
&
\begin{aligned}
&\left\| \vf \int_{-\infty}^x e^{2\Lambda} \vf^2 dy \right\|_{L_T^2 H_x^1} \\
&\lesssim \left\| \dx \vf \int_{-\infty}^x e^{2\Lambda} \vf^2 dy \right\|_{L_{T,x}^2} + \| e^{2\Lambda} \vf^3 \|_{L_{T,x}^2} +  \left\| \vf \int_{-\infty}^x e^{2\Lambda} \vf^2 dy \right\|_{L_{T,x}^2} \\
&\lesssim \| e^{2\Lambda} \|_{L_{T,x}^{\infty}} \left( \| \vf \|_{L_T^{\infty} L_x^2}^2 \| \dx \vf \|_{L_T^{\infty} L_x^2} + \| \vf \|_{L_{T,x}^6}^3 + \| \vf \|_{L_T^{\infty} L_x^2}^3 \right) \\
&\lesssim \| e^{2\Lambda} \|_{L_{T,x}^{\infty}} \| \lr{\dx} \vf \|_{X_T}^3,
\end{aligned}
\\
&
\begin{aligned}
\| u^3 \vf \|_{L_T^2 H_x^1}
&\lesssim \| u^3 \dx \vf \|_{L_{T,x}^2} + \| u^2 \dx u \vf \|_{L_{T,x}^2} + \| u^3 \vf \|_{L_{T,x}^2} \\
&\lesssim \| u \|_{L_{T,x}^{\infty}}^3  \| \dx \vf \|_{L_T^{\infty} L_x^2} + \| u \|_{L_{T,x}^{\infty}}^2 \| \dx u \|_{L_T^{\infty} L_x^2} \| \vf \|_{L_{T,x}^{\infty}} + \| u \|_{L_{T,x}^{\infty}}^3 \| \vf \|_{L_T^{\infty} L_x^2} \\
&\lesssim \| \lr{\dx} u \|_{X_T}^3 \| \lr{\dx} \vf \|_{X_T}.
\end{aligned}
\end{align*}
Accordingly, \eqref{est:lin} and \eqref{product} imply that
\begin{equation} \label{Phiest2}
\begin{aligned}
&\| \lr{\dx} \Phi_{u_0}^{(2)} (u,\vf) \|_{X_T} - C_1 \| \vf_0 \|_{H^1} \\
&\lesssim T^{\frac{1}{2}} \Bigg( \| e^{\Lambda} \vf \dx \vf \| _{L_T^2 H_x^1} + \| e^{\Lambda} u \vf^2 \|_{L_T^2 H_x^1} + \| u^2 \dx \vf \|_{L_T^2 H_x^1} + \left\| \vf \int_{-\infty}^x e^{2\Lambda} \vf^2  dy \right\|_{L_T^2 H_x^1} \\
&\qquad \qquad + \| u^3 \vf \|_{L_T^2 H_x^1} \Bigg) \\
&\lesssim T^{\frac{1}{2}} \left( e^{2\| \Lambda \|_{L_{T,x}^{\infty}}} + 1 \right) \left( \| (u,\vf) \|_{Y_T}^2 + \| (u,\vf) \|_{Y_T}^4 \right).
\end{aligned}
\end{equation}

Since $\Psi_{u_0}^{(1)} (u,\vf)$ satisfies
\[
\left(\partial_t+\frac{1}{3}\partial_x^3\right)\Psi_{u_0}^{(1)} (u,\vf)
=c_1 e^{\Lambda} u \left( \dx \vf + c_1 u \vf \right) + c_2 e^{2\Lambda} \vf^2, 
\]
the fundamental theorem of calculus shows
\begin{equation} \label{eq:gauge0}
\begin{aligned}
&\int_{-\infty}^x \Psi_{u_0}^{(1)} (u,\vf) (t,y) dy - \int_{-\infty}^x u_0(y) dy \\
&= \int_0^t \frac{d}{d \tau} \int_{-\infty}^x \Psi_{u_0}^{(1)} (u,\vf) (\tau,y) dy d\tau \\
&= - \frac{1}{3} \int_0^t \dx^2 \Psi_{u_0}^{(1)} (u,\vf) (\tau,x)  d\tau + \int_0^t \int_{-\infty}^x \left( c_1 e^{\Lambda} u \left( \dx \vf + c_1 u \vf \right) + c_2 e^{2\Lambda} \vf^2 \right) dy d\tau \\
&= - \frac{1}{3} \int_0^t \dx^2 \Psi_{u_0}^{(1)} (u,\vf) (\tau,x)  d\tau + c_1 \int_0^t e^{\Lambda} u \vf d\tau - c_1 \int_0^t \int_{-\infty}^x c_1 e^{\Lambda} \dx u \vf dy d\tau \\
&\qquad + c_2 \int_0^t \int_{-\infty}^x e^{2\Lambda} \vf^2 dy d\tau,
\end{aligned}
\end{equation}
which leads to the following:
\begin{equation}
\begin{aligned}
&\left\| \int_{-\infty}^x \Psi_{u_0}^{(1)} (u,\vf) (t,y) dy \right\|_{L_{T,x}^{\infty}} - \left\| \int_{-\infty}^x u_0(y) dy \right\|_{L_x^{\infty}} \\
&\lesssim \| \dx^2 \Phi_{u_0}^{(1)} (u,\vf) \|_{L_x^{\infty} L_T^1} + \| e^{\Lambda} u \vf \|_{L_x^{\infty} L_T^1} + \| e^{\Lambda} \dx u \vf \|_{L_{T,x}^1} + \| e^{2\Lambda} \vf^2 \|_{L_{T,x}^1} \\
&\lesssim T^{\frac{1}{2}} \Big( \| \dx^2 \Phi_{u_0}^{(1)} (u,\vf) \|_{L_x^{\infty} L_T^2} + \| e^{\Lambda} \|_{L^{\infty}_{T,x}} \| u \|_{L_{T,x}^{\infty}} \| \vf \|_{L_T^{\infty} L_x^{\infty}} \\
&\qquad + \| e^{\Lambda} \|_{L^{\infty}_{T,x}} \| \dx u \|_{L_T^{\infty} L_x^2} \| \vf \|_{L_T^{\infty} L_x^2} + \| e^{2\Lambda} \|_{L^{\infty}_{T,x}} \| \vf \|_{L_T^{\infty} L_x^2}^2 \Big) \\
&\le C_0 T^{\frac{1}{2}} \left( \| \dx \Phi_{u_0}^{(1)} (u,\vf) \|_{X_T} + \| e^{\Lambda} \|_{L^{\infty}_{T,x}} \| \lr{\dx} u \|_{X_T} \| \lr{\dx} \vf \|_{X_T} + \| e^{2\Lambda} \|_{L^{\infty}_{T,x}} \| \vf \|_{X_T}^2 \right).
\end{aligned}
\label{justLambda}
\end{equation}

Therefore, \eqref{Phiest1}, \eqref{Phiest2}, and \eqref{justLambda} yield that
\begin{equation} \label{Phiest3}
\begin{aligned}
&\| \Phi_{u_0} (u,\vf) \|_{Y_T} \\
&\le 2C_1 \left( \| u_0 \|_{\mathcal{X}^1} + \| \vf_0 \|_{H^1} \right) + C_2 T^{\frac{1}{2}} e^{2 |c_1| \| (u,\vf) \|_{Y_T}} \left( \| (u,\vf) \|_{Y_T}^2 + \| (u,\vf) \|_{Y_T}^4 \right)
\end{aligned}
\end{equation}
provided that $C_0 T^{\frac{1}{2}} < \frac{1}{2}$.
A similar calculation leads to the estimate for the difference.

Here, we set a closed ball $B_T$ of $Y_T$ by
\[
B_T := \left\{ (u,\vf) \in Y_T \colon \| (u,\vf) \|_{Y_T} \le 3C_1 ( \| u_0 \|_{\mathcal{X}^1} + \| \vf_0 \|_{H^1}) \right\}.
\]
Then, $\Phi_{u_0}$ is a contraction mapping on $B_T$ if $T$ is small depending only on $\| u_0 \|_{\mathcal{X}^1}$ and $\| \vf_0 \|_{H^1}$.

If $u_0 \in \mathcal{X}^2$, we have $\vf_0 = e^{-c_1 \int_{-\infty}^x u_0 (y)dy} \dx u_0 \in H^1 (\R)$.
Because $(u,\vf)$ is a solution to \eqref{eq:u}--\eqref{eq:vf}, the equation $\vf (t,x) = e^{-c_1 \int_{-\infty}^x  u(t,y) dy} \dx u(t,x)$ holds, which implies the well-posedness in $\mathcal{X}^2$ of the Cauchy problem for \eqref{KdV}.

For the reader's convenience, we give the proof of this fact.
Let $w := \dx u - e^{\Lambda} \vf$.
By \eqref{eq:u}, a direct calculation shows that
\begin{align*}
\Lc \dx u &= e^{\Lambda} \left( c_1 u \dx^2 \vf + c_1 \dx u \dx \vf + 2 c_1^2 u^2 \dx \vf + 2c_1^2 u \dx u \vf  + c_1^3 u^3 \vf \right) \\
&\quad + 2 c_2 e^{2\Lambda} \left( \vf \dx \vf + c_1 u \vf^2 \right), \\
\int_{-\infty}^x \Lc u dy &= - c_1 \int_{-\infty}^x e^{\Lambda} \dy u \vf dy +c_1 e^{\Lambda} u \vf + c_2 \int_{-\infty}^x e^{2\Lambda} \vf^2 dy \\
&= -c_1 \int_{-\infty}^x e^{\Lambda} \vf w dy + c_1 e^{\Lambda} u \vf - (c_1-c_2) \int_{-\infty}^x e^{2\Lambda} \vf^2 dy.
\end{align*}
From \eqref{gauge} and \eqref{eq:vf}, we have
\begin{align*}
e^{-\Lambda} \Lc (e^{\Lambda} \vf)
&= \Lc \vf + \left( c_1^2 u \dx u + \frac{c_1^3}{3} u^3 + \Lc \Lambda \right) \vf + \left(  c_1 \dx u + c_1^2 u^2 \right) \dx \vf + c_1 u \dx^2 \vf \\
&= 2c_2 e^{\Lambda} \vf ( \dx \vf + c_1 u \vf) + 2c_1^2 u^2 \dx \vf + c_1^3 u^3 \vf \\
&\quad + \left( c_1^2 u \dx u -c_1^2 \int_{-\infty}^x e^{\Lambda} \vf w dy + c_1^2 e^{\Lambda} u \vf \right) \vf + c_1 \dx u \dx \vf + c_1 u \dx^2 \vf
\end{align*}
Accordingly, we obtain
\[
\Lc w
= c_1^2 e^{\Lambda} u \vf w + c_1^2 e^{\Lambda} \vf \int_{-\infty}^x e^{\Lambda} \vf w dy.
\]
The same calculation as in \eqref{Phiest2} leads to
\[
\| w \|_{X_T}
\lesssim T^{\frac{1}{2}} e^{2 \| \Lambda \|_{L_{T,x}^{\infty}}}  \left( \| u \|_{X_T} + \| \vf \|_{X_T} \right) \| \vf \|_{X_T} \| w \|_{X_T}.
\]
By $w(0)=0$, the standard continuity argument shows that $w(t) =0$ for $|t| \le T$.
Therefore, we obtain that $\vf (t,x) = e^{-c_1 \int_{-\infty}^x  u(t,y) dy} \dx u(t,x)$ for $|t| \le T$.

\section{Well-posedness for \eqref{KdV} in $\mathcal{X}^1$} \label{S:WPX1}

We first consider the special case $c_2=0$, because the general case is a bit complicated.
In \S \ref{S:WPX11}, we show the well-posedness in $\mathcal{X}^1$ under $c_2=0$.
In \S \ref{S:WPX12}, we observe the persistency of regularity for $c_2=0$.
Finally, in \S \ref{S:WPX13}, we prove Theorem \ref{WP} without $c_2=0$.

\subsection{Proof of Theorem \ref{WP} under $c_2=0$} \label{S:WPX11}

Let $c_2=0$ and $u_0 \in \mathcal{X}^1$.
Then, there exists a sequence $\{ u_{0,n} \} \subset \mathcal{X}^2$ such that $u_{0,n}$ converges to $u_0$ in $\mathcal{X}^1$.
Without loss of generality, we may assume that $\| u_{0,n} \|_{\mathcal{X}^1} \le 2 \| u_0 \|_{\mathcal{X}^1}$ holds for any $n \in \N$.
By the well-posedness in $\mathcal{X}^2$, there exist $T_n>0$ and the solution $u_n \in C([-T_n,T_n]; \mathcal{X}^2)$, where $T_n$ depends on $\| u_{0,n} \|_{\mathcal{X}^2}$.

Set $\Lambda_n (t,x) := c_1 \int_{-\infty}^x u_n(t,y) dy$ and $\vf_n := e^{-\Lambda_n} \dx u_n$.
First, we observe the following bound.

\begin{lem} \label{lem:unbound}
\[
\| u_n \|_{L_x^2 L_T^{\infty}}
\lesssim e^{\frac{3}{2} \| \Lambda_n \|_{L_{T,x}^{\infty}}} \left( \| u_n \|_{X_T} + \| u_n \|_{X_T}^2 + \| \vf_n \|_{X_T}^2 \right).
\]
\end{lem}

\begin{proof}
The low frequency part is easily handed:
\[
\| P_1 u_n \|_{L_x^2 L_T^{\infty}}
\lesssim \| \lr{\dx}^{-\frac{3}{4}-\eps} P_1 u_n \|_{L_x^2 L_T^{\infty}}
\lesssim \| u_n \|_{X_T}. 
\]
We use the Littlewood-Paley decomposition to estimate the high frequency part:
\begin{equation} \label{unLP}
\begin{aligned}
\| P_{>1} u_n \|_{L_x^2 L_T^{\infty}}
&\lesssim \| P_{>1} \lr{\dx}^{-1} (e^{\Lambda_n} \vf_n) \|_{L_x^2 L_T^{\infty}} \\
&\lesssim \sum_{N_1,N_2 \in 2^{\N_0}} \| \lr{\dx}^{-1} (P_{N_1} e^{\Lambda_n} P_{N_2} \vf_n) \|_{L_x^2 L_T^{\infty}}.
\end{aligned}
\end{equation}
For $N_1 \gtrsim N_2$, we have
\begin{align*}
\| \lr{\dx}^{-1} (P_{N_1} e^{\Lambda_n} P_{N_2} \vf_n) \|_{L_x^2 L_T^{\infty}}
&\lesssim \| P_{N_1} e^{\Lambda_n} P_{N_2} \vf_n \|_{L_x^2 L_T^{\infty}} \\
&\lesssim N_1^{-\frac{1}{4}+\eps} \| \dx P_{N_1} e^{\Lambda_n} \|_{L_{T,x}^{\infty}} \| \lr{\dx}^{-\frac{3}{4}-\eps} P_{N_2} v_n \|_{L_x^2 L_T^{\infty}} \\
&\lesssim N_1^{-\frac{1}{4}+\eps} \| e^{\Lambda_n} \|_{L_{T,x}^{\infty}} \| u_n \|_{L_{T,x}^{\infty}} \| v_n \|_{X_T} \\
&\lesssim N_1^{-\frac{1}{4}+\eps} e^{\frac{3}{2} \| \Lambda_n \|_{L_{T,x}^{\infty}}} \| u_n \|_{X_T}^{\frac{1}{2}} \| v_n \|_{X_T}^{\frac{3}{2}}.
\end{align*}
Here, we have used the Gagliardo-Nirenberg type inequality in the last inequality as follows:
\begin{equation} \label{unGN}
\| u_n \|_{L_{T,x}^{\infty}}
\lesssim \| u_n \|_{L_T^{\infty}L_x^2}^{\frac{1}{2}} \| \dx u_n \|_{L_T^{\infty}L_x^2}^{\frac{1}{2}}
\lesssim \| e^{\Lambda_n} \|_{L_{T,x}^{\infty}}^{\frac{1}{2}} \| u_n \|_{L_T^{\infty} L_x^2}^{\frac{1}{2}} \| \vf_n \|_{L_T^{\infty} L_x^2}^{\frac{1}{2}}.
\end{equation}
When $N_1 \ll N_2$, because the frequency of the product of the two functions is around $N_2$, we have
\begin{align*}
\| \lr{\dx}^{-1} (P_{N_1} e^{\Lambda_n} P_{N_2} \vf_n) \|_{L_x^2 L_T^{\infty}}
&\lesssim N_2^{-\frac{1}{4}+\eps} \| e^{\Lambda_n} \|_{L_{T,x}^{\infty}} \| \lr{\dx}^{-\frac{3}{4}-\eps} \vf_n \|_{L_x^2 L_T^{\infty}} \\
&\lesssim N_2^{-\frac{1}{4}+\eps} \| e^{\Lambda_n} \|_{L_{T,x}^{\infty}} \| \vf_n \|_{X_T}.
\end{align*}
Hence, by using $(N_1+N_2)^{-\frac{1}{4}+\eps}$, we can sum up the summation with respect to $N_1$ and $N_2$ in \eqref{unLP}.
Therefore, we obtain the desired bound.
\end{proof}

Lemma \ref{lem:unbound} and \eqref{unGN} yield that
\begin{align}
& \label{estun1}
\begin{aligned}
&\| u_n \dx \vf_n \|_{L_{T,x}^2} + \| u_n^2 \vf_n \|_{L_{T,x}^2} \\
&\lesssim \| u_n \|_{L_x^2 L_{T}^{\infty}} \| \dx \vf_n \|_{L_x^{\infty} L_{T}^2} + \| u_n \|_{L_{T,x}^6}^2 \| \vf_n \|_{L_{T,x}^6} \\
&\lesssim e^{\frac{3}{2} \| \Lambda_n \|_{L_{T,x}^{\infty}}} \left( \| u_n \|_{X_T} + \| \lr{\dx} u_n \|_{X_T}^2+ \| \vf_n \|_{X_T}^2 \right) \| \vf_n \|_{X_{T}},
\end{aligned}
\\
& \label{estun2}
\begin{aligned}
&\| u_n^2 \dx \vf_n \|_{L_{T,x}^2} + \| u_n^3 \vf_n \|_{L_{T,x}^2} + \left\| \vf_n \int_{-\infty}^x e^{2 \Lambda_n} \vf_n^2 dy \right\|_{L_{T,x}^2} \\
&\lesssim \| u_n \|_{L_x^2 L_{T}^{\infty}} \| u_n \|_{L_{T,x}^{\infty}} \| \dx \vf_n \|_{L_x^{\infty} L_{T}^2} + \| u_n \|_{L_{T,x}^8}^3 \| \vf_n \|_{L_{T,x}^8} \\
&\qquad + \| e^{2\Lambda_n} \|_{L_{T,x}^{\infty}} \| \vf_n \|_{L_{T}^{\infty} L_x^2}^3\\
&\lesssim e^{2 \| \Lambda_n \|_{L_{T,x}^{\infty}}} \left( \| u_n \|_{X_T}^2 + \| \vf_n \|_{X_T}^2 + \| u_n \|_{X_T}^3 \right) \| \vf_n \|_{X_T}.
\end{aligned}
\end{align}
Since \eqref{unGN} yields that
\begin{equation} \label{uvest}
\begin{aligned}
\| u_n \vf_n \|_{L_x^{\infty} L_{T}^1}
&\lesssim \| u_n P_1 \vf_n \|_{L_x^{\infty} L_{T}^1} + \| u_n  P_{>1} \vf_n \|_{L_x^{\infty} L_{T}^1} \\
&\lesssim T^{\frac{1}{2}} \| u_n \|_{L_{T,x}^{\infty}} \left( \| P_1 \vf_n \|_{L_{T,x}^{\infty}} + \| P_{>1} \vf_n \|_{L_x^{\infty} L_T^2} \right) \\
&\lesssim T^{\frac{1}{2}} \| e^{\Lambda_n} \|_{L_{T,x}^{\infty}}^{\frac{1}{2}} \| u_n \|_{X_T}^{\frac{1}{2}} \| \vf_n \|_{X_T}^{\frac{3}{2}},
\end{aligned}
\end{equation}
by \eqref{eq:gauge0}, we have
\begin{equation}
\label{estun3}
\begin{aligned}
&\left\| \int_{-\infty}^x u_n (t,y) dy \right\|_{L_{T_n,x}^{\infty}} - \left\| \int_{-\infty}^x u_{0,n} (y) dy \right\|_{L_x^{\infty}} \\
&\lesssim \| \dx^2 u_n \|_{L_x^{\infty} L_{T_n}^1} + \| u_n \dx u_n \|_{L_x^{\infty} L_{T_n}^1} + \| (\dx u_n)^2 \|_{L_{T_n,x}^1} \\
&\lesssim \| e^{\Lambda_n} \|_{L_{T_n,x}^{\infty}} \Big( \| \dx \vf_n \|_{L_x^{\infty} L_{T_n}^1} + \| u_n \vf_n \|_{L_x^{\infty} L_{T_n}^1} + \| e^{\Lambda_n}  \|_{L_{T_n,x}^{\infty}} \| \vf_n^2 \|_{L_{T_n,x}^1} \Big) \\
&\lesssim T_n^{\frac{1}{2}} e^{2 \| \Lambda_n \|_{L_{T_n,x}^{\infty}}} \Big( 1+ \| u_n \|_{X_{T_n}} + \| \vf_n \|_{X_{T_n}} \Big) \| \vf_n \|_{X_{T_n}}.
\end{aligned}
\end{equation}

We set
\[
\| u \|_{Z_T} := \| u \|_{X_T} + \left\| e^{-c_1 \int_{-\infty}^x u(t,y) dy} \dx u \right\|_{X_T} + \left\| \int_{-\infty}^x u(t,y) dy \right\|_{L_{T,x}^{\infty}}.
\]
Because $u_n$ and $\vf_n$ satisfy \eqref{eq:u}, \eqref{eq:vf} with $c_2=0$, the estimates \eqref{est:lin}, \eqref{estun1}, \eqref{estun2}, and \eqref{estun3} yield that
\begin{equation} \label{apriori00}
\| u_n \|_{Z_{T_n}} \le C_1 \| u_{0,n} \|_{\mathcal{X}^1} + C_2 T_n^{\frac{1}{2}} e^{\frac{5}{2} |c_1| \| u_n \|_{Z_{T_n}}} \| u_n \|_{Z_{T_n}} \left( 1 + \| u_n \|_{Z_{T_n}}^3 \right).
\end{equation}

For simplicity, we set
\[
\| u \|_{\wt{Y}_T}
= \| u \|_{Z_T} + \left\| \lr{\dx} \left( e^{-c_1 \int_{-\infty}^x u(t,y) dy} \dx u \right) \right\|_{X_T}.
\]
Since Lemma \ref{lem:unbound} and \eqref{unGN} lead to
\begin{equation}
\label{estun4}
\begin{aligned}
&\left\| \dx \left\{ e^{-\Lambda_n} \left( c_1^2 u_n^2 \dx^2 u_n + c_1^2 \dx u_n \int_{-\infty}^x (\dy u_n)^2 dy - \frac{c_1^3}{3} u_n^3 \dx u_n \right) \right\} \right\|_{L_{T,x}^2} \\
&\lesssim \| u_n^2 \dx^2 \vf_n \|_{L_{T,x}^2} + \| u_n^3 \dx \vf_n \|_{L_{T,x}^2} + \left\| \dx \vf_n \int_{-\infty}^x e^{2 \Lambda_n} \vf_n^2 dy \right\|_{L_{T,x}^2} \\
&\qquad + \| e^{\Lambda_n} \|_{L_{T,x}^{\infty}} \left( \| u_n \vf_n \dx \vf_n \|_{L_{T,x}^2} + \| u_n^2 \vf_n^2 \|_{L_{T,x}^2} + \| e^{\Lambda_n} \|_{L_{T,x}^{\infty}} \| \vf_n^3 \|_{L_{T,x}^2} \right) \\
&\lesssim \| u_n \|_{L_x^2 L_{T}^{\infty}} \| u_n \|_{L_{T,x}^{\infty}} \| \dx^2 \vf_n \|_{L_x^{\infty} L_{T}^2} + \| u_n \|_{L_{T,x}^8}^3 \| \dx \vf_n \|_{L_{T,x}^8} \\
&\qquad + \| e^{2\Lambda_n} \|_{L_{T,x}^{\infty}} \| \dx \vf_n \|_{L_{T}^{\infty} L_x^2} \| \vf_n \|_{L_{T}^{\infty} L_x^2}^2\\
&\quad + e^{2 \| \Lambda_n \|_{L_{T,x}^{\infty}}} \left( \| u_n \|_{L_{T,x}^6} \| \vf_n \|_{L_{T,x}^6} \| \dx \vf_n \|_{L_{T,x}^6} + \| u_n \|_{L_{T,x}^8}^2 \| \vf_n \|_{L_{T,x}^8}^2 + \| \vf_n \|_{L_{T,x}^6}^3 \right) \\
&\lesssim e^{2 \| \Lambda_n \|_{L_{T,x}^{\infty}}} \| \lr{\dx} \vf_n \|_{X_T} \left( \| u_n \|_{Z_T}^2 + \| u_n \|_{Z_T}^3 \right),
\end{aligned}
\end{equation}
we have
\[
\| u_n \|_{\wt{Y}_{T_n}}
\le C_1 \| u_{0,n} \|_{\mathcal{X}^2} + C_2 T_n^{\frac{1}{2}} e^{\frac{5}{2} |c_1| \| u_n \|_{Z_{T_n}}} \| u_n \|_{\wt{Y}_{T_n}} \left( 1+ \| u_n \|_{Z_{T_n}}^3 \right).
\]

Here, we set
\[
T^{\ast} := \frac{1}{10} \left( C_2 e^{10 |c_1| C_1 \| u_0 \|_{\mathcal{X}^1}} \left\{ 1 + (4 C_1 \| u_0 \|_{\mathcal{X}^1})^3 \right\} \right)^{-2},
\]
which is independent of $n$.
By $\| u_{0,n} \|_{\mathcal{X}^1} \le 2 \| u_0 \|_{\mathcal{X}^1}$, the continuity argument shows
\[
\| u_n \|_{Z_{T_n^{(0)}}} \le 3C_1 \| u_0 \|_{\mathcal{X}^1}, \quad
\| u_n \|_{\wt{Y}_{T_n^{(0)}}} \le 3C_1 \| u_{0,n} \|_{\mathcal{X}^2},
\]
where $T_n^{(0)} := \min (T_n, T^{\ast})$.
Then, Theorem \ref{WP} yields that there exists $\rho_n$ depending on $\| u_0 \|_{\mathcal{X}^1}$ and $\| u_{0,n} \|_{\mathcal{X}^2}$ such that $u_n$ satisfies \eqref{KdV} on $[-(T_n + \rho_n), T_n + \rho_n]$.
Because we can apply the estimates \eqref{estun1}, \eqref{estun2}, \eqref{estun3}, and \eqref{estun4} as long as $u_n$ is a solution to \eqref{KdV}, we obtain
\begin{align*}
\| u_n \|_{Z_{T_n+\rho_n}}
\le & C_1 \| u_{0,n} \|_{\mathcal{X}^1} \\
&\quad + C_2 (T_n+\rho_n)^{\frac{1}{2}} e^{\frac{5}{2} |c_1| \| u_n \|_{Z_{T_n+\rho_n}}} \| u_n \|_{Z_{T_n+\rho_n}} \left( 1 + \| u_n \|_{Z_{T_n+\rho_n}}^3 \right), \\
\| u_n \|_{\wt{Y}_{T_n+\rho_n}}
\le & C_1 \| u_{0,n} \|_{\mathcal{X}^2} \\
&\quad + C_2 (T_n+\rho_n)^{\frac{1}{2}} e^{\frac{5}{2} |c_1| \| u_n \|_{Z_{T_n+\rho_n}}} \| u_n \|_{\wt{Y}_{T_n+\rho_n}} \left( 1+ \| u_n \|_{Z_{T_n+\rho_n}}^3 \right).
\end{align*}
By setting $T_n^{(1)} := \min (T_n+\rho_n, T^{\ast})$, these bounds show that
\[
\| u_n \|_{Z_{T_n^{(1)}}} \le 3C_1 \| u_0 \|_{\mathcal{X}^1}, \quad
\| u_n \|_{\wt{Y}_{T_n^{(1)}}} \le 3C_1 \| u_{0,n} \|_{\mathcal{X}^2}.
\]
By repeating this procedure $k$-times, we can extend this bound to that for $T_n^{(k)} := \min (T_n+k\rho_n, T^{\ast})$ and $k \in \N$.
In particular, because there exists an integer $k_n$ such that $T^{(k_n)}_{n} = T^{\ast}$, we obtain
\begin{equation} \label{boundu_n}
\| u_n \|_{Z_{T^{\ast}}} \le 3C_1 \| u_0 \|_{\mathcal{X}^1}, \quad
\| u_n \|_{\wt{Y}_{T^{\ast}}} \le 3C_1 \| u_{0,n} \|_{\mathcal{X}^2}
\end{equation}
for any $n \in \N$.

Next, we consider the estimate for the difference.
By \eqref{eq:gauge0}, \eqref{uvest}, \eqref{boundu_n}, and taking  $T^{\ast}$ small if necessary, we have
\begin{align*}
&\| \Lambda_n - \Lambda_m \|_{L_{T^{\ast},x}^{\infty}} - \left\| \int_{-\infty}^x (u_{0,n}(y) - u_{0,m}(y)) dy \right\|_{L_x^{\infty}} \\
&\lesssim \| \dx^2 u_n - \dx^2 u_m \|_{L_x^{\infty} L_{T^{\ast}}^1} + \| u_n \dx u_n - u_m \dx u_m\|_{L_x^{\infty} L_{T^{\ast}}^1} \\
&\qquad + \| (\dx u_n)^2 - (\dx u_m)^2 \|_{L_{T^{\ast},x}^1} \\
&\lesssim \| e^{\Lambda_n} - e^{\Lambda_m} \|_{L_{T^{\ast},x}^{\infty}} \| \dx \vf_n \|_{L_x^{\infty} L_{T^{\ast}}^1}
+ \| e^{\Lambda_m} \|_{L_{T^{\ast},x}^{\infty}} \| \dx \vf_n - \dx \vf_m \|_{L_x^{\infty} L_{T^{\ast}}^1} \\
&\qquad + \| e^{\Lambda_n} - e^{\Lambda_m} \|_{L_{T^{\ast},x}^{\infty}} \| u_n \vf_n \|_{L_x^{\infty} L_{T^{\ast}}^1}
+ \| e^{\Lambda_m} \|_{L_{T^{\ast},x}^{\infty}} \| u_n \vf_n - u_m \vf_m \|_{L_x^{\infty} L_{T^{\ast}}^1} \\
&\qquad + \| e^{2\Lambda_n} - e^{2\Lambda_m} \|_{L_{T^{\ast},x}^{\infty}} \| \vf_n^2 \|_{L_x^{\infty} L_{T^{\ast}}^1}
+ \| e^{2\Lambda_m} \|_{L_{T^{\ast},x}^{\infty}} \| \vf_n^2 - \vf_m^2 \|_{L_x^{\infty} L_{T^{\ast}}^1} \\
&\le \frac{1}{2} \left( \| \Lambda_n - \Lambda_m \|_{L_{T^{\ast},x}^{\infty}} + \| u_n - u_m \|_{X_{T^{\ast}}} + \| \vf_n - \vf_m \|_{X_{T^{\ast}}} \right).
\end{align*}
Because the remaining cases are similarly handled, we obtain
\[
\| \Lambda_n - \Lambda_m \|_{L_{T^{\ast},x}^{\infty}} + \| u_n - u_m \|_{X_{T^{\ast}}} + \| \vf_n - \vf_m \|_{X_{T^{\ast}}}
\lesssim \| u_{0,n} - u_{0,m} \|_{\mathcal{X}^1}.
\]
Therefore, $\{ u_n \}$ is a Cauchy sequence and the limit $u$ is in $C([-T^{\ast}, T^{\ast}]; \mathcal{X}^1)$.
Hence, we conclude that \eqref{KdV} is well-posed in $\mathcal{X}^1$ if $c_2=0$.

\subsection{Persistence of regularity} \label{S:WPX12}

Let $c_2=0$, $s \ge 1$, and $u_0 \in \mathcal{X}^s$.
The well-posedness in \S \ref{S:WPX11} says that there exist the time $T>0$ and the solution $u \in C([-T,T];\mathcal{X}^1)$.
We prove that the solution has regularity, i.e., $u \in C([-T,T]; \mathcal{X}^s)$, where $T$ depends only on $\| u_0 \|_{\mathcal{X}^1}$.
For simplicity, we set $r:=s-1 \ge 0$, $\Lambda (t,x) := c_1 \int_{-\infty}^x u(t,y) dy$, and $\vf := e^{-\Lambda} \dx u$.
Moreover, we define
\[
\| f \|_{X^r_T} :=
\left\| \lr{\dx}^r f \right\|_{X_T} + \left\| |\dx|^{r+\frac{1}{8}} f \right\|_{L_T^8 L_x^4} + \left\| \dx^{k+1} f \right\|_{L_x^{\frac{4}{r-k}} L_T^{\frac{4}{2-(r-k)}}},
\]
where $k$ is the integer satisfying $k< r \le k+1$.
Note that the third term on the right hand side is meaningless if $r \in \N_0$.
Indeed, for $k \in \N_0$ and $0<T<1$, we have $\| \dx^{k+1} f \|_{L_{T,x}^4} \lesssim \| \lr{\dx}^{k+1} f \|_{X_T}$.

We apply Lemmas \ref{lem:Str} and \ref{lem:Kato} and Stein's interpolation theorem \cite{Ste56} as in \cite{KPV93} to obtain
\[
\| |\dx|^{\theta} \lp (t) u_0 \|_{L_x^{\frac{4}{1-\theta}} L_T^{\frac{4}{1+\theta}}} \lesssim \| u_0 \|_{L^2}
\]
for $0<T<1$ and $0<\theta<1$.
Hence, by \eqref{est:lin}, we have
\begin{equation} \label{est:linr}
\| \lp (t) u_0 \|_{X^r_T} \le C_1 \| u_0 \|_{H^r}
\end{equation}
for $0<T<1$.
We also use the following norms:
\begin{align*}
&\| u \|_{\wt{X}_T^r} := \| u \|_{X_T^r} + \left\| e^{-c_1 \int_{-\infty}^x u(t,y) dy} \dx u \right\|_{X_T^r}, \\
&\| u \|_{Z_T^s} := \| u \|_{\wt{X}_T^{s-1}} + \left\| \int_{-\infty}^x u(t,y) dy \right\|_{L_{T,x}^{\infty}}.
\end{align*}

We observe a product estimate in the Sobolev space, while similar estimates are known (see, for example, Theorem 4 of \S 4.6.2 in \cite{RunSic96}, Theorem A.1 in \cite{MacOka15}, and Lemma 2.2 in \cite{MacOka17}).

\begin{lem} \label{lem:hombi}
For $r \ge 0$, we have
\[
\| f g \|_{H^r}
\lesssim
\| f \|_{H^r} \| g \|_{L^{\infty}} + \| f \|_{H^{r-[r]}} \| g \|_{\dot{H}^{[r]+1}},
\]
where $[r]$ means the largest integer less than or equal to $r$.
\end{lem}

\begin{proof}
We use the paraproduct decomposition:
\begin{equation} \label{para}
f g = \sum_{\substack{N_1,N_2 \in 2^{\N_0} \\ N_1 \gg N_2}} P_{N_1} f P_{N_2} g + \sum_{\substack{N_1,N_2 \in 2^{\N_0} \\ N_1 \lesssim N_2}} P_{N_1} f P_{N_2} g
=: \text{I} + \text{II}.
\end{equation}
We note that the first term on the right hand side is written as follows:
\[
\lr{\dx}^r \sum_{\substack{N_1,N_2 \in 2^{\N_0} \\ N_1 \gg N_2}} P_{N_1} f P_{N_2} g
= \frac{1}{\sqrt{2\pi}} \iint_{\R^2} e^{ix (\xi+\eta)} \sigma (\xi,\eta) \wh{\lr{\dx}^r f}(\xi) \wh{g}(\eta) d\xi d\eta,
\]
where $\sigma (\xi, \eta) := \frac{\lr{\xi+\eta}^r}{\lr{\xi}^r} \phi \left( \frac{\eta}{\xi} \right)$ and $\phi$ is a smooth function with $\supp \phi \subset [-\frac{1}{2},\frac{1}{2}]$.
A direct calculation shows that
\[
| \partial_{\xi}^{\alpha} \partial_{\eta}^{\beta} \sigma (\xi,\eta)| \lesssim_{\alpha,\beta} (|\xi|+|\eta|)^{-\alpha-\beta}
\]
for $(\xi,\eta) \in \R^2 \setminus \{ (0,0) \}$ and $\alpha, \beta \in \N_0$.
Accordingly, we can apply Coifman-Meyer's Fourier multiplier theorem (see \cite{CoiMey78}) to obtain
\[
\| I \|_{H^r}
\lesssim \| f \|_{H^r} \| g \|_{L^{\infty}}.
\]

The second term on the right hand side of \eqref{para} is calculated as follows:
\begin{align*}
\| \text{II} \|_{H^r}
&\lesssim \sum_{\substack{N_1,N_2 \in 2^{\N_0} \\ N_1 \lesssim N_2}} N_2^r \| P_{N_1} f P_{N_2} g \|_{L^2} \\
&\lesssim
\sum_{\substack{N_1 \in 2^{\N_0} \\ N_1 \sim 1}} \left\| P_{N_1} f \right\|_{L^2} \left\| P_{1} g \right\|_{L^{\infty}}\\
&\quad + \sum_{N_1 \in 2^{\N_0}} \sum_{\substack{N_2 \in 2^{\N} \\ N_1 \lesssim N_2}} N_1^{-r+[r]+\frac{1}{2}} N_2^{r-[r]-1} \left\| P_{N_1} \lr{\dx}^{r-[r]} f \right\|_{L^2} \left\| P_{N_2} |\dx|^{[r]+1} g \right\|_{L^2} \\
&\lesssim \| f \|_{L^2} \| g \|_{L^{\infty}} + \| f \|_{H^{r-[r]}} \| g \|_{\dot{H}^{[r]+1}}.
\end{align*}
\end{proof}

Thanks to
\[
\| e^{\Lambda} \|_{\dot{H}^{k}}
\lesssim \| e^{\Lambda} \|_{L^{\infty}} \| u \|_{H^{k-1}}  \left( 1+ \| u \|_{L^2 \cap H^{k-2}}^{k-1} \right)
\]
for $k \in \N$, Lemma \ref{lem:hombi} leads to
\begin{equation} \label{product2}
\| e^{\Lambda} f \|_{H^r}
\lesssim \| e^{\Lambda} \|_{L^{\infty}} \left( \| f \|_{H^r} + \| u \|_{H^{[r]}} \left( 1+ \| u \|_{L^2 \cap H^{[r]-1}}^{[r]} \right) \| f \|_{H^{r-[r]}} \right).
\end{equation}

We show a generalized version of Lemma \ref{lem:unbound}.

\begin{lem} \label{lem:unbound2}
For $r \ge 0$, we have
\[
\left\| \lr{\dx}^r u \right\|_{L_x^2 L_T^{\infty}}
\lesssim e^{2 \| \Lambda \|_{L_{T,x}^{\infty}}} \left( \| u \|_{\wt{X}_T^{\max (r-\frac{1}{4}+2\eps,0)}} + \| u \|_{\wt{X}_T^{\max (r-\frac{1}{4}+2\eps,0)}}^{[r]+3} \right).
\]
\end{lem}

\begin{proof}
As in \eqref{unLP}, we have
\begin{equation} \label{unLP2}
\begin{aligned}
\left\| P_{>1} \lr{\dx}^r u \right\|_{L_x^2 L_T^{\infty}}
&\lesssim \left\| P_{>1} \lr{\dx}^{r-1} (e^{\Lambda} \vf) \right\|_{L_x^2 L_T^{\infty}} \\
&\lesssim \sum_{N_1,N_2 \in 2^{\N_0}} \left\| P_{>1} \lr{\dx}^{r-1} (P_{N_1} e^{\Lambda} P_{N_2} \vf) \right\|_{L_x^2 L_T^{\infty}}.
\end{aligned}
\end{equation}
For $N_1 \gtrsim N_2$, we have
\begin{align*}
&\left\| P_{>1} \lr{\dx}^{r-1} (P_{N_1} e^{\Lambda} P_{N_2} \vf) \right\|_{L_x^2 L_T^{\infty}} \\
&\lesssim \left\| \dx^{[r]} (P_{N_1} e^{\Lambda} P_{N_2} \vf) \right\|_{L_x^2 L_T^{\infty}} \\
&\lesssim N_1^{-\frac{1}{4}+\eps} \| \dx^{[r]+1} P_{N_1} e^{\Lambda} \|_{L_{T,x}^{\infty}} \left\| \lr{\dx}^{-\frac{3}{4}-\eps} P_{N_2} \vf \right\|_{L_x^2 L_T^{\infty}} \\
&\lesssim N_1^{-\frac{1}{4}+\eps} \| \dx^{[r]} (e^{\Lambda} u) \|_{L_{T,x}^{\infty}} \| \vf \|_{X_T}.
\end{align*}
When $[r] \ge 1$, Sobolev's embedding and \eqref{product2} yield that
\begin{align*}
\| \dx^{[r]} (e^{\Lambda} u) \|_{L_{T,x}^{\infty}}
&\lesssim \| \dx^{[r]-1} (e^{\Lambda} u^2) \|_{L_{T,x}^{\infty}} + \| \dx^{[r]-1} (e^{2\Lambda} \vf ) \|_{L_{T,x}^{\infty}} \\
&\lesssim \| e^{\Lambda} u^2 \|_{H^{[r]-\frac{1}{2}+\eps}} + \| e^{2\Lambda} \vf \|_{H^{[r]-\frac{1}{2}+\eps}} \\
&\lesssim \| e^{\Lambda} \|_{L_{T,x}^{\infty}} \left( \| u \|_{L_T^{\infty} H^{r-\frac{1}{2}+\eps}}^2 + \| u \|_{L_T^{\infty} H^{r-\frac{1}{2}+\eps}}^{[r]+2} \right) \\
&\quad + \| e^{2\Lambda} \|_{L_{T,x}^{\infty}} \| \vf \|_{L_T^{\infty} H^{r-\frac{1}{2}+\eps}} \left( 1 + \| u \|_{L_T^{\infty} H^{[r]-1}}^{[r]} \right) \\
&\lesssim e^{2 \| \Lambda \|_{L_{T,x}^{\infty}}} \left( \| u \|_{\wt{X}_T^{r-\frac{1}{2}+\eps}} + \| u \|_{\wt{X}_T^{r-\frac{1}{2}+\eps}}^{[r]+2} \right).
\end{align*}
When $[r]=0$, \eqref{unGN} yields that
\[
\| \dx^{[r]} (e^{\Lambda} u) \|_{L_{T,x}^{\infty}}
\lesssim \|e^{\Lambda}\|_{L_{T,x}^{\infty}}\|u\|_{L_{T,x}^{\infty}}
\lesssim e^{\frac{3}{2}\|\Lambda \|_{L_{T,x}^{\infty}}}\|u\|_{X_T}^{\frac{1}{2}}\| \vf \|_{X_T}^{\frac{1}{2}}.
\]
Hence, we have
\begin{align*}
&\left\| P_{>1} \lr{\dx}^{r-1} (P_{N_1} e^{\Lambda} P_{N_2} \vf) \right\|_{L_x^2 L_T^{\infty}} \\
&\lesssim N_1^{-\frac{1}{4}+\eps} e^{2 \| u \|_{Z_T^0}} \left( \| u \|_{\wt{X}_T^{\max(r-\frac{1}{2}+\eps,0)}}^2 + \| u \|_{\wt{X}_T^{\max(r-\frac{1}{2}+\eps,0)}}^{[r]+3} \right)
\end{align*}
for $N_1 \gtrsim N_2$.

When $N_1 \ll N_2$, the frequency of the product of the two functions is around $N_2$.
For $0 \le r<\frac{1}{4}-\eps$, we have
\begin{align*}
\left\| \lr{\dx}^{r-1} (P_{N_1} e^{\Lambda} P_{N_2} \vf) \right\|_{L_x^2 L_T^{\infty}}
&\lesssim N_2^{r-\frac{1}{4}+\eps} \| P_{N_1} e^{\Lambda} \|_{L_{T,x}^{\infty}} \left\| \lr{\dx}^{-\frac{3}{4}-\eps} P_{N_2} \vf \right\|_{L_x^2 L_T^{\infty}} \\
&\lesssim N_2^{r-\frac{1}{4}+\eps} \| e^{\Lambda} \|_{L_{T,x}^{\infty}} \| \vf \|_{X_T}.
\end{align*}
For $r \ge \frac{1}{4}-\eps$, we have
\begin{align*}
\left\| \lr{\dx}^{r-1} (P_{N_1} e^{\Lambda} P_{N_2} \vf) \right\|_{L_x^2 L_T^{\infty}}
&\lesssim N_2^{-\eps} \| P_{N_1} e^{\Lambda} \|_{L_{T,x}^{\infty}} \left\| \lr{\dx}^{r-1+\eps} P_{N_2} \vf  \right\|_{L_x^2 L_T^{\infty}} \\
&\lesssim N_2^{-\eps} \| e^{\Lambda} \|_{L_{T,x}^{\infty}} \left\| \lr{\dx}^{r-\frac{1}{4}+2\eps}\vf \right\|_{X_T}.
\end{align*}
Hence, we can sum up the summation with respect to $N_1$ and $N_2$ in \eqref{unLP2}.
Therefore, we obtain the desired bound.
\end{proof}

Let $\wt{r} := r-[r]$.
The fractional Leibniz rule (see Appendix in \cite{KPV93}), Lemma \ref{lem:unbound2}, and an interpolation argument yield that
\begin{align*}
&\left\| |\dx|^{r} (u \dx \vf) \right\|_{L_{T,x}^2} \\
&\lesssim \sum_{k=0}^{[r]} \left\| |\dx|^{\wt{r}} (\dx^{[r]-k} u \dx^{k+1} \vf) \right\|_{L_{T,x}^2} \\
&\lesssim \sum_{k=0}^{[r]} \bigg( \left\| \dx^{[r]-k} u |\dx|^{\wt{r}} \dx^{k+1} \vf \right\|_{L_{T,x}^2} + \left\| |\dx|^{\wt{r}} \dx^{[r]-k} u \right\|_{L_x^{\frac{4}{2-\wt{r}}} L_T^{\frac{4}{\wt{r}}}} \left\| \dx^{k+1} \vf \right\|_{L_x^{\frac{4}{\wt{r}}} L_T^{\frac{4}{2-\wt{r}}}} \bigg) \\
&\lesssim \sum_{k=0}^{[r]} \bigg( \left\| \dx^{[r]-k} u \right\|_{L_x^2 L_T^{\infty}} \left\| |\dx|^{\wt{r}} \dx^{k+1} \vf \right\|_{L_x^{\infty} L_T^2} \\
&\qquad \qquad + \left\| |\dx|^{\wt{r}} \dx^{[r]-k} u \right\|_{L_x^2L_T^{\infty}}^{1-\wt{r}} \left\| |\dx|^{\wt{r}} \dx^{[r]-k} u \right\|_{L_{T,x}^4}^{\wt{r}} \| u \|_{Z^r_T} \bigg) \\
&\lesssim e^{2 \| \Lambda \|_{L_{T,x}^{\infty}}} \left( \| u \|_{\wt{X}_T^{\max(r-\frac{1}{8},0)}} + \| u \|_{\wt{X}_T^{\max(r-\frac{1}{8},0)}}^{[r]+3} \right) \| u \|_{\wt{X}^r_T}.
\end{align*}
Because Sobolev's embedding and \eqref{product2} imply that
\begin{align*}
\| \dx^k u \|_{L_x^{\infty}}
\lesssim \| \dx^{k-1} (e^{\Lambda} \vf ) \|_{L_x^{\infty}}
\lesssim \| e^{\Lambda} \vf \|_{H_x^{k-\frac{1}{2}+\eps}}
\lesssim \| e^{\Lambda} \|_{L_x^{\infty}}  \| \vf \|_{H_x^{k-\frac{1}{2}+\eps}} \left( 1 + \| \vf \|_{H_x^{k-1}}^k \right)
\end{align*}
for $k \in \N$, the same calculation as above leads to
\begin{align*}
&\left\| |\dx|^{r} (u^2 \dx \vf) \right\|_{L_{T,x}^2} \\
&\lesssim \sum_{k=0}^{[r]} \left\| |\dx|^{\wt{r}} (\dx^{[r]-k} (u^2) \dx^{k+1} \vf) \right\|_{L_{T,x}^2} \\
&\lesssim \sum_{k=0}^{[r]} \bigg( \left\| \dx^{[r]-k} (u^2) |\dx|^{\wt{r}} \dx^{k+1} \vf \right\|_{L_{T,x}^2} + \left\| |\dx|^{\wt{r}} \dx^{[r]-k} (u^2) \right\|_{L_x^{\frac{4}{2-\wt{r}}} L_T^{\frac{4}{\wt{r}}}} \left\| \dx^{k+1} \vf \right\|_{L_x^{\frac{4}{\wt{r}}} L_T^{\frac{4}{2-\wt{r}}}} \bigg) \\
&\lesssim e^{3 \| \Lambda \|_{L_{T,x}^{\infty}}} \left( \| u \|_{\wt{X}_T^{\max(r-\frac{1}{8},0)}}^2 + \| u \|_{\wt{X}_T^{\max(r-\frac{1}{8},0)}}^{[r]+4} \right) \| u \|_{\wt{X}^r_T}.
\end{align*}

Lemma \ref{lem:hombi} and \eqref{product2} show that
\begin{align*}
&\left\| \vf \int_{-\infty}^x e^{2 \Lambda} \vf^2 dy \right\|_{L_T^2 H_x^r} \\
&\lesssim \| \vf \|_{L_T^{\infty} H_x^r} \| e^{2 \Lambda} \|_{L_{T,x}^{\infty}} \| \vf \|_{L_T^{\infty} L_x^2}^2 + \| \vf \|_{L_T^{\infty} H_x^{\wt{r}}} \left\| e^{2 \Lambda} \vf^2 \right\|_{L_T^2 H_x^{[r]}} \\
&\lesssim \| e^{2\Lambda} \|_{L_{T,x}^{\infty}} \left( \| u \|_{\wt{X}_T^{\max(r-1,0)}}^2 + \| u \|_{\wt{X}_T^{\max(r-1,0)}}^{[r]+2} \right) \| u \|_{\wt{X}_T^r}.
\end{align*}

Because the remaining terms on the right hand side of \eqref{eq:u} and \eqref{eq:vf} with $c_2=0$ more easily handed, the estimates \eqref{estun3} and \eqref{est:linr} yield that
\begin{equation*}
\| u \|_{Z_T^s} \le C_1 \| u_0 \|_{\mathcal{X}^s} + C_2 T^{\frac{1}{2}} e^{3 |c_1| \| u \|_{Z_T^1}} \left( \| u \|_{Z_T^{\max(s-\frac{7}{8},1)}} + \| u \|_{Z_T^{\max(s-\frac{7}{8},1)}}^{[s]+3} \right) \| u \|_{Z^s_T}.
\end{equation*}
The persistence property follows from this a priori bound with a standard continuity argument.

\subsection{Proof of Theorem \ref{WP} without $c_2 = 0$} \label{S:WPX13}

The first term on the right hand side of \eqref{eq:vf} causes some technical difficulty, because it has a quadratic term with one derivative.
However, by using a gauge transformation, we cancel out this term.
As in the previous subsection, the well-posedness is reduced to show an a priori estimate as \eqref{apriori00}.

Let $\Xi (t,x) = c_2 \int_{-\infty}^x u(t,y)dy$.
Then, \eqref{KdV}, \eqref{gauge}, and \eqref{LLambda} yield
\[
e^{\Xi} \Lc \left( e^{-\Xi} u \right)
= (c_1-c_2) u \dx^2u - (c_1-2c_2)c_2 u^2 \dx u + (c_1-c_2)c_2 u \int_{-\infty}^x (\dy u)^2 dy - \frac{c_2^3}{3} u^4.
\]
Set $\uf := e^{-\Xi} u$.
Since
\begin{align*}
\dx u &= e^{\Xi} ( \dx \uf + c_2 u \uf) = e^{\Xi} \dx \uf + c_2 e^{2\Xi} \uf^2, \\
\dx^2 u &= e^{\Xi} ( \dx^2 \uf + c_2 u \dx \uf) + c_2 e^{2\Xi} ( 2 \uf \dx \uf + 2c_2 u \uf^2) \\
& = e^{\Xi} \dx^2 \uf + 3 c_2 e^{2\Xi} \uf \dx \uf + 2c_2^2 e^{3\Xi} \uf^3,
\end{align*}
we have
\begin{equation} \label{uf1}
\begin{aligned}
\Lc \uf &= (c_1-c_2) \uf (e^{\Xi} \dx^2 \uf + 3 c_2 e^{2\Xi} \uf \dx \uf + 2c_2^2 e^{3\Xi} \uf^3) \\
&\quad - (c_1-2c_2) c_2 e^{\Xi} \uf^2 (e^{\Xi} \dx \uf + c_2 e^{2\Xi} \uf^2) \\
&\quad + (c_1-c_2)c_2 \uf \int_{-\infty}^x (e^{\Xi} \dy \uf + c_2 e^{2\Xi} \uf^2)^2 dy - \frac{c_2^3}{3} e^{3\Xi} \uf^4 \\
&= (c_1-c_2) e^{\Xi} \uf  \dx^2 \uf + (2c_1-c_2)c_2 e^{2\Xi} \uf^2 \dx \uf \\
&\quad + (c_1-c_2)c_2 \uf \int_{-\infty}^x (e^{\Xi} \dy \uf + c_2 e^{2\Xi} \uf^2)^2 dy + \left( c_1- \frac{c_2}{3} \right) c_2^2 e^{3\Xi} \uf^4.
\end{aligned}
\end{equation}
A direct calculation shows that
\begin{equation} \label{duf}
\Lc \dx \uf = (c_1-c_2) e^{\Xi} \uf \dx^3 \uf + (c_1-c_2) e^{\Xi} \dx \uf \dx^2 \uf + \mathcal{N},
\end{equation}
where $\mathcal{N}$ is a linear combination of
\[
e^{2\Xi} \uf^2 \dx^2 \uf, \quad
e^{2\Xi} \uf (\dx \uf)^2, \quad
e^{3\Xi} \uf^3 \dx \uf, \quad
\dx \uf \int_{-\infty}^x (e^{\Xi} \dy \uf + c_2 e^{2\Xi} \uf^2)^2 dy, \quad
e^{4\Xi} \uf^5.
\]
Moreover, let $\Theta := (c_1-c_2) \int_{-\infty}^x (e^{\Xi} \uf )(t,y) dy = (c_1-c_2) \int_{-\infty}^x u (t,y) dy$ and $\vf := e^{-\Theta} \dx \uf$.
Because 
\begin{align*}
\Lc \Theta
&= (c_1-c_2) c_1 u \dx u - (c_1-c_2)^2 \int_{-\infty}^x (\dy u)^2 dy \\
&= (c_1-c_2) c_1 e^{\Xi} \uf ( e^{\Xi} \dx \uf + c_2 e^{2\Xi} \uf^2) - (c_1-c_2)^2 \int_{-\infty}^x (e^{\Xi} \dy \uf + c_2 e^{2\Xi} \uf^2)^2 dy,
\end{align*}
\eqref{gauge} and \eqref{duf} imply that $\Lc \vf$ is equal to a linear combination of
\[
e^{2\Xi} \uf^2 \dx \vf, \quad
e^{2\Xi+\Theta} \uf \vf^2, \quad
e^{3\Xi} \uf^3 \vf, \quad
\vf \int_{-\infty}^x (e^{\Xi+\Theta} \vf + c_2 e^{2\Xi} \uf^2)^2 dy, \quad
e^{4\Xi-\Theta} \uf^5.
\]
In addition,  \eqref{uf1} is written as follows:
\begin{align*}
\Lc \uf
&= (c_1-c_2) e^{\Xi+\Theta} \uf  \dx \vf + c_1^2 e^{2\Xi+\Theta} \uf^2 \vf + (c_1-c_2)c_2 \uf \int_{-\infty}^x (e^{\Xi+\Theta} \vf + c_2 e^{2\Xi} \uf^2)^2 dy \\
&\quad  + \left( c_1 - \frac{c_2}{3} \right) c_2^2 e^{3\Xi} \uf^4.
\end{align*}

Here, we define the norm
\begin{align*}
\| u \|_{\wt{Z}_T} :=& \left\| e^{-c_2\int_{-\infty}^x u(t,y) dy} u \right\|_{X_T} + \left\| e^{-(c_1-c_2)\int_{-\infty}^x u(t,y) dy} \dx \left( e^{-c_2\int_{-\infty}^x u(t,y) dy} u \right) \right\|_{X_T} \\
&\quad + \left\| \int_{-\infty}^x u(t,y) dy \right\|_{L_{T,x}^{\infty}}.
\end{align*}
Then, \eqref{est:lin}, \eqref{estun1}, and \eqref{estun2} yield that
\begin{equation} \label{apriori000}
\| u \|_{\wt{Z}_{T}} \le C_1 \| u_{0} \|_{\mathcal{X}^1} + C_2 T^{\frac{1}{2}} e^{5 (|c_1|+|c_2|) \| u \|_{\wt{Z}_{T}}} \| u \|_{\wt{Z}_{T}} \left( 1 + \| u \|_{\wt{Z}_{T}}^4 \right)
\end{equation}
as long as $u$ is a solution to \eqref{KdV}.
Hence, the same argument as in \S \ref{S:WPX11} shows that the existence time $T$ depends only on $\| u_0 \|_{\mathcal{X}^1}$.
Moreover, \eqref{KdV} is well-posed in $\mathcal{X}^1$.
Because the persistency follows from the same argument as in \S \ref{S:WPX12}, we omit the details here.

\section{Well-posedness for the quadratic KdV-type equation}

In this section, we consider the Cauchy problem for the semi-linear KdV-type equation with quadratic nonlinearity.
Let $u$ be a solution to \eqref{KdV2}.
Then, $\dx u$ and $\dx^2u$ satisfy the following equations:
\begin{align}
\label{KdV'}
\Lc \dx u &= (c_1+2c_2) \dx u \dx^2 u + \left( c_1 u + c_3 \dx u \right) \dx^3 u + c_3 (\dx^2 u)^2 + 2 c_4 \dx^2 u \dx^3 u, \\
\label{KdV''}
\Lc \dx^2 u &= (c_1+2c_2) (\dx^2 u)^2 + \left( 2(c_1+c_2) \dx u + 3 c_3 \dx^2 u \right) \dx^3u \\
&\quad + \left( c_1 u + c_3 \dx u + 2c_4 \dx^2 u \right) \dx^4u + 2c_4 (\dx^3 u)^2. \notag
\end{align}

Set $\J := 2 c_4 \dx u$ and $w := e^{-\J} \dx^2u$.
Then, \eqref{gauge}, \eqref{KdV'}, and \eqref{KdV''} yield that
\begin{equation} \label{KdV'''}
\Lc w
= (c_1u +c_3 \dx u) \dx^2 w + \mathcal{N}_1,
\end{equation}
where $\mathcal{N}_1$ is a linear combination of forms
\[
f_1 \dx w, \quad
f_1 f_2 \dx w, \quad
e^{-\J} f_1f_2, \quad
e^{-\J} f_1f_2f_3, \quad
e^{-\J} f_1f_2f_3f_4
\]
for $f_j \in \{ u, \dx u, e^{\J} w \}$.

Let $\K := c_1 \int_{-\infty}^x u dy + c_3 u$ and $\wf := e^{-\K} \dx w$.
Because
\[
\Lc \int_{-\infty}^x u(t,y) dy
= c_1 u \dx u - (c_1-c_2) \int_{-\infty}^x (\dy u)^2 dy + c_4 \int_{-\infty}^x (\dy^2 u)^2 dy,
\]
\eqref{gauge} and \eqref{KdV'''} imply that $\Lc \wf$ is equal to a linear combination of forms
\begin{align*}
&f_1\dx \wf, \ f_1 f_2 \dx \wf,  \quad
\left(\int_{-\infty}^x (\dy u)^2 dy\right) \wf, \quad
\left(\int_{-\infty}^x e^{2\J} w^2 dy\right) \wf, \\
&e^{-\J-\K} g_1 g_2, \quad
e^{-\J-\K} g_1 g_2 g_3, \quad
e^{-\J-\K} g_1 g_2 g_3 g_4, \quad
e^{-\J-\K} g_1 g_2 g_3 g_4 g_5.
\end{align*}
for $f_j \in \{ u, \dx u, e^{\J} w \}$, $g_k \in \{ u, \dx u, e^{\J} w, e^{\J+\K} \wf \}$.
Moreover, \eqref{KdV2}, \eqref{KdV'}, and \eqref{KdV'''} are written as follows:
\begin{align*}
\Lc u &= c_1 e^{\J} u w + c_2 (\dx u)^2 + c_3 e^{\J} \dx u w + c_4 e^{2\J} w^2, \\
\Lc \dx u &= (c_1+2c_2) e^{\J} \dx u w + (c_1u+c_3 \dx u) e^{\J} ( e^{\K} \wf + 2c_4 w^2) + c_3 e^{2\J} w^2 \\
&\quad + 2c_4 e^{2\J} w ( e^{\K} \wf + 2c_4 w^2), \\
\Lc w &= (c_1u+c_3 \dx u) e^{\K} \dx \wf + \wt{\mathcal{N}}_1,
\end{align*}
where $\wt{\mathcal{N}}_1$ is a linear combination of forms
\[
e^{\K} f_1 \wf, \quad
e^{\K} f_1 f_2 \wf, \quad
e^{-\J} f_1f_2, \quad
e^{-\J} f_1f_2f_3, \quad
e^{-\J} f_1f_2f_3f_4
\]
for $f_j \in \{ u, \dx u, e^{\J} w \}$.
Hence, we can apply the contraction mapping theorem as in \S \ref{S:WPX0} to obtain well-posedness in $\mathcal{X}^4$ of \eqref{KdV2}.

We define the norm as follows:
\begin{align*}
\| u \|_{\mathcal{Z}_T} :=& \| u \|_{X_T} + \| \dx u \|_{X_T} + \left\| e^{-c_4 \dx u} \dx^2 u \right\|_{X_T} \\
&\quad+ \left\| e^{-c_1\int_{-\infty}^x u(t,y) dy - c_3 u} \dx \left( e^{-c_4 \dx u} \dx^2 u \right) \right\|_{X_T} + \left\| c_1 \int_{-\infty}^x u(t,y) dy \right\|_{L_{T,x}^{\infty}}.
\end{align*}
Because
\begin{align*}
\| \J \|_{L_{T,x}^{\infty}}
&\le 2|c_4| \| \dx u \|_{L_{T,x}^{\infty}}
\lesssim \| u \|_{L_T^{\infty} H^2}, \\
\| \K \|_{L_{T,x}^{\infty}}
&\le |c_1| \left\| \int_{-\infty}^x u(t,y) dy \right\|_{L_{T,x}^{\infty}} + |c_3| \| u \|_{L_{T,x}^{\infty}}
\lesssim \| u \|_{L_T^{\infty} \mathcal{X}^1}, 
\end{align*}
\eqref{est:lin} and a similar calculation as in \eqref{estun1} and \eqref{estun2} yield that
\[
\| u \|_{\mathcal{Z}_{T}} \le C_1 \| u_{0} \|_{\mathcal{X}^1} + C_2 T^{\frac{1}{2}} e^{C_3 \| u \|_{\mathcal{Z}_{T}}} \| u \|_{\mathcal{Z}_{T}} \left( 1 + \| u \|_{\mathcal{Z}_{T}}^4 \right)
\]

When $c_4=0$, we set
\begin{align*}
\| u \|_{\mathcal{Z}'_T} & := \| u \|_{X_T} + \left\| e^{-c_3 u} \dx u \right\|_{X_T} + \left\| e^{-c_1\int_{-\infty}^x u(t,y) dy} \dx \left( e^{-c_3 u} \dx u \right) \right\|_{X_T} \\
&\quad + \left\| c_1 \int_{-\infty}^x u(t,y) dy \right\|_{L_{T,x}^{\infty}}.
\end{align*}
Then, the same argument as above shows that
\[
\| u \|_{\mathcal{Z}'_{T}} \le C_1 \| u_{0} \|_{\mathcal{X}^1} + C_2 T^{\frac{1}{2}} e^{C_3 \| u \|_{\mathcal{Z}'_{T}}} \| u \|_{\mathcal{Z}'_{T}} \left( 1 + \| u \|_{\mathcal{Z}'_{T}}^4 \right).
\]

\begin{rmk}
When $c_1=0$, the boundedness of primitives is not necessary, because $\int_{-\infty}^x u(t,y) dy$ disappears in $\| u \|_{\mathcal{Z}_T}$ and $\| u \|_{\mathcal{Z}'_T}$.
\end{rmk}

\section{Irregular flow maps}

\subsection{On the condition for initial data}

For $c_1 \neq 0$, Pilod \cite{Pil08} proved that the flow map fails to be twice differentiable in $H^s(\R)$ for any $s \in \R$.
Here, we briefly observe that our result does not contradict to Pilod's result.

For simplicity, we consider \eqref{KdV} with $c_1 \neq 0$ and $c_2=0$.
Pilod put the following sequence of the initial data:
\[
u_{0,N} := \mathcal{F}^{-1} \left[ N \bm{1}_{[-N^{-2},N^{-2}]} + N^{-s+1} \bm{1}_{[-N-N^{-2},-N+N^{-2}] \cup [N-N^{-2},N+N^{-2}]} \right]
\]
for any $N \ge 1$.
Then, $\| u_{0,N} \|_{H^s} \lesssim 1$.

If $\xi_1 \in [N-N^{-2},N+N^{-2}]$ and $\xi-\xi_1 \in [-N^{-2},N^{-2}]$, then $\xi \in [N-2N^{-2}, N+2N^{-2}]$ and
\[
|\xi^3-(\xi-\xi_1)^3-\xi_1^3| = 3 |\xi \xi_1 (\xi-\xi_1)| \lesssim 1.
\]
Accordingly, for $0<T \ll 1$, we have
\begin{align*}
&\left\| \int_0^t \lp (t-t') \left( \lp (t') u_{0,N}(x) \lp (t') \dx^2 u_{0,N} (x) \right) dt' \right\|_{L_T^{\infty} H^s} \\
&\gtrsim T \left\| N^{-s+2} \mathcal{F}^{-1} \left[ \bm{1}_{[N-N^{-2},N+N^{-2}]} \right] \right\|_{L_T^{\infty} H^s}
\gtrsim T N,
\end{align*}
which shows the flow map fails to be twice differentiable in $H^s(\R)$.

By a simple calculation, the initial datum is written as follows:
\[
u_{0,N}(x) = \sqrt{\frac{2}{\pi}} \left( 1 + 2N^{-s} \cos Nx \right) N \frac{\sin N^{-2}x}{x}.
\]
Since
\begin{align*}
\int_{-\infty}^{\infty} u_{0,N} (y) dy
&= \sqrt{\frac{2}{\pi}} N \left( \int_{-\infty}^{\infty} \frac{\sin y}{y} dy + 2 N^{-s} \int_{-\infty}^{\infty} \frac{\sin y \cos N^3 y}{y} dy \right) \\
&= \sqrt{2\pi} N,
\end{align*}
this sequence is not bounded in $\mathcal{X}^s$.
In other words, we can avoid the worst interaction because of $\sup_{x \in \R} \left| \int_{-\infty}^x u_0 (y) dy \right| <\infty$.

\subsection{Not $C^2$ in $\mathcal{X}^s$}

For simplicity, we assume that $c_1=1$ and $c_2=0$.
We set
\[
u_{0,N} := \mathcal{F}^{-1} \left[ N^{-s+\frac{a}{2}} \bm{1}_{[-N-N^{-a},-N+N^{-a}] \cup [N-N^{-a},N+N^{-a}]} \right]
\]
for any $N \gg 1$ and $a>0$.
Then, $\| u_{0,N} \|_{H^s} \lesssim 1$.
Since
\[
u_{0,N} (x) = 2 \sqrt{\frac{2}{\pi}} N^{-s+\frac{a}{2}} \frac{\sin N^{-a} x}{x} \cos Nx,
\]
a direct calculation shows
\begin{align*}
\int_{-\infty}^x u_{0,N} (y) dy
&=\sqrt{\frac{2}{\pi}} N^{-s+\frac{a}{2}} \int_{-\infty}^x \left\{ \frac{\sin (N+N^{-a})y}{y} - \frac{\sin (N-N^{-a})y}{y} \right\} dy \\
&=\sqrt{\frac{2}{\pi}} N^{-s+\frac{a}{2}} \int_{(N-N^{-a})x}^{(N+N^{-a})x} \frac{\sin y}{y} dy.
\end{align*}
The mean value theorem for integrals yields
\[
\sup_{x \in \R} \left| \int_{-\infty}^x u_{0,N}(y) dy \right|
\lesssim N^{-s-\frac{a}{2}-1}.
\]
Therefore, $\{ u_{0,N} \}$ is a bounded sequence in $\mathcal{X}^s$ provided that $s>-\frac{a}{2}-1$.

On the other hand, for $0<T \ll 1$, we have
\begin{align*}
&\sup_{t \in [-T,T], x \in \R} \left| \int_{-\infty}^x \int_0^t \lp (t-t') \left( \lp (t') u_{0,N}(y) \lp (t') \dx^2 u_{0,N} (y) \right) dt' dy \right| \\
&\gtrsim \sup_{t \in [-T,T]} \left| \mathcal{F} \left[ \int_0^t \lp (t-t') \left( \lp (t') u_{0,N}(x) \lp (t') \dx^2 u_{0,N} (x) \right) dt' \right] (0) \right| \\
&\gtrsim T \left| \mathcal{F} \left[ u_{0,N} \dx^2 u_{0,N} \right] (0) \right|
\gtrsim T N^{-2s+2},
\end{align*}
which shows the flow map fails to be twice differentiable in $\mathcal{X}^s$ for $s<1$.

\section*{Acknowledgment}
This work was supported by JSPS KAKENHI Grant Numbers JP16K17624 and JP17K14220, 
Program to Disseminate Tenure Tracking System from the Ministry of Education, Culture, Sports, Science and Technology, 
and the DFG through the CRC 1283 ``Taming uncertainty and
profiting from randomness and low regularity in analysis, stochastics
and their
applications.''

\end{document}